\documentclass[onecolumn]{article}

\usepackage[utf8]{inputenc}                                     %(only for the pdftex engine)
\usepackage[big]{dgruyter_NEW}
 
\usepackage{amsfonts,amsmath,latexsym,amssymb} 					% these packages are required
\usepackage{amsthm}                									
\usepackage{hyperref}
\usepackage{kbordermatrix}									    % Fancy matrices
\usepackage[usenames,dvipsnames,svgnames,table]{xcolor}
\usepackage{enumitem}

\newtheorem{thm}{Theorem}[section]
\newtheorem{prob}[thm]{Problem}

\theoremstyle{definition}			                			% Subseq. Thm Style
\newtheorem{ex}[thm]{Example}

\numberwithin{equation}{section}		            			% Number eqns within section

\newcommand{\sr}[1]{\rho\left(#1\right)}						% spectral radius
\newcommand{\sig}[1]{\sigma \left( #1 \right)}				    % multi-set of eigenvalues
\newcommand{\dg}[1]{\Gamma \left( #1 \right)}					% directed graph (gamma}
  
\begin{document}
\articletype{Short Note{\hfill}Open Access}

\author*[1]{Pietro Paparella}
\affil[1]{University of Washington Bothell, E-mail: pietrop@uw.edu}

\title{\huge Spectrally Perron Polynomials and the Cauchy-Ostrovosky Theorem}
\runningtitle{Spectrally Perron Polynomials}

  %\subtitle{...}

\begin{abstract}
{In this note, we simplify the statements of theorems attributed to Cauchy and Ostrovsky and give proofs of each theorem via combinatorial and nonnegative matrix theory. We also show that each simple sufficient condition in each statement is also necessary in its respective case. In addition, we introduce the notion of a \emph{spectrally Perron polynomial} and pose a problem that appeals to a wide mathematical audience.}
\end{abstract}
  
%  \classification[PACS]{}
 % \communicated{...}
 % \dedication{...}

\linenumbers
  \journalname{Spec. Matrices}
\DOI{DOI}
  \startpage{1}
  \received{...}
  \revised{...}
  \accepted{...}

  \journalyear{...}
  \journalvolume{...}
%  \journalissue{1}
\maketitle

%-----------------------------------------------------------------------------------------------------------------------------------------------------------------------------------------------------------------------------
\section{Introduction}
%-----------------------------------------------------------------------------------------------------------------------------------------------------------------------------------------------------------------------------

Because the roots of a polynomial coincide with the eigenvalues of its \emph{companion matrix}, many classical results on the geometry of polynomials can be obtained with relative ease via matricial methods: for instance, Wilf \cite{w1961} used the Perron-Frobenius theroem to derive \emph{Cauchy's bound} on the maximum modulus of any given root and Bell \cite{b1965} used the Gershgorin Theorem to derive the same bound along with other bounds and a result due to Walsh (see \cite[\S 5.6]{hj2013} for a bounds obtained via matrix norms). 

The purpose of this work is to give proofs of results attributed to Cauchy and Ostrovsky via combinatorial matrix theory and nonnegative matrix theory. In the process of doing so, we simplify the statement of each of these results by giving a single, simple sufficient condition that is shown to be necessary (respectively) in each theorem. In addition, we introduce the notion of a \emph{spectrally Perron polynomial} and pose a problem that appeals to a wide mathematical audience.

%-----------------------------------------------------------------------------------------------------------------------------------------------------------------------------------------------------------------------------
\section{Notation \& Background}
%-----------------------------------------------------------------------------------------------------------------------------------------------------------------------------------------------------------------------------
 
A real matrix is called \emph{nonnegative} (\emph{positive}) if it is entrywise nonnegative (respectively, positive) matrix. 

A \emph{directed graph} (or simply \emph{digraph}) $\Gamma = (V,E)$ consists of a finite, nonempty set $V$ of \emph{vertices}, together with a set $E \subseteq V \times V$ of \emph{arcs}. For an $n$-by-$n$ matrix $A$, the \emph{digraph} of $A$, denoted by $\Gamma = \dg{A}$, has vertex set $V = \{ 1, \dots, n \}$ and arc set $E = \{ (i, j) \in V \times V : a_{ij} \neq 0\}$. 

A digraph $\Gamma$ is called \emph{strongly connected} if for any two distinct vertices $i$ and $j$ of $\Gamma$, there is a walk in $\Gamma$ from $i$ to $j$ (following \cite{br1991}, we consider every vertex of $V$ as strongly connected to itself). A strong digraph is \emph{primitive} if the greatest common divisor of all its cycle-lengths is one, otherwise it is \emph{imprimitive}. 

For $n \geq 2$, an $n$-by-$n$ matrix $A$  is called \emph{reducible} if there exists a permutation matrix $P$ such that
\begin{align*}
P^\top A P =
\begin{bmatrix}
A_{11} & A_{12} \\
0 & A_{22}
\end{bmatrix},
\end{align*}
where $A_{11}$ and $A_{22}$ are nonempty square matrices. If $A$ is not reducible, then A is called \emph{irreducible}. It is well-known that a matrix $A$ is irreducible if and only if $\dg{A}$ is strongly connected (see, e.g., \cite[Theorem 3.2.1]{br1991} or \cite[Theorem 6.2.24]{hj2013}). 

Borrowing terminology from \cite{lm1995}, we call a real matrix $A$ \emph{spectrally Perron} if there is a simple eigenvalue $\rho > 0$ such that 
\begin{equation}
\rho > |\lambda|,~\forall \lambda \in \sig{A}. \label{strict}
\end{equation}
If $\rho$ is a simple positive eigenvalue and the inequality \eqref{strict} is not strict, then we call $A$ \emph{weakly spectrally Perron}.

An irreducible nonnegative matrix is called \emph{primitive} if its digraph is primitive; otherwise it is \emph{imprimitive}. If $A \geq 0$, then $A$ is primitive if and only if some positive integral power of $A$ is positive \cite[Theorem 3.4.4]{br1991}. Thus, $A$ is spectrally Perron if $A$ is primitive \cite{h1981, jt2004,n2006}. 

Given an $n$-by-$n$ matrix $A$, the \emph{characteristic polynomial of $A$}, denoted by $\chi_A$, is defined by $\chi_A = \det{(tI - A)}$. The \emph{companion matrix} $C = C_p$ of a monic polynomial $p(t) = t^n + \sum_{k=1}^{n} c_{k} t^{n - k}$ is the $n$-by-$n$ matrix defined by
\[ C = 
\left[\begin{array}{rr}
0 & I \\
-c_n & -c
\end{array} \right], \]
where $c = [c_{n-1}~\cdots~c_1]$. It is well-known that $\chi_C = p$. Notice that $C$ is irreducible if and only if $c_n \neq 0$. 

Finally, we call a polynomial $p$ \emph{(weakly) spectrally Perron} if its companion matrix is (weakly) spectrally Perron.

%-----------------------------------------------------------------------------------------------------------------------------------------------------------------------------------------------------------------------------
\section{The Cauchy-Ostrovosky Theorem}
%-----------------------------------------------------------------------------------------------------------------------------------------------------------------------------------------------------------------------------

Before we state the proofs, we introduce some brief notation. In this section, $p(t) = t^n - c_1 t^{n - 1} - \cdots -  c_n$, where $c_k \geq 0$, for every $k$. For $k = 1, \dots, n$, let 
\begin{equation*}
i_k :=
\left\{
\begin{array}{r l}
k, & c_k \neq 0 \\
0, & c_k = 0.
\end{array}
\right.
\end{equation*}
Finally, let $d := \gcd{\left(i_1,\dots,i_n \right)}$, $\mathcal{I} := \{ i_1, \dots, i_n \}$, and $\mathcal{C} := \{ c_1,\dots, c_n \}$.

The following result (c.f., \cite[Theorem 1.1.3]{p2010}) is attributed to Cauchy and a proof-sketch via nonnegative matrix theory is given in \cite[Theorem 1]{w1961}; we give a proof via nonnegative matrix theory and show that the sufficient condition is also necessary.

%-----------------------------------------------------------------------------------------------------------------------------------------------------------------------------------------------------------------------------
\begin{thm}[Cauchy]
\label{thm:cauchy}
The polynomial $p$ is weakly spectrally Perron if and only if $d \neq 0$. 
\end{thm}

\begin{proof}
Denote by $J_n (\lambda)$ the $n$-by-$n$ Jordan block with eigenvalue $\lambda$. If $d \neq 0$, then at least one of the elements of $\mathcal{I}$ is positive, i.e., at least one of the elements of $\mathcal{C}$ is positive. If $\ell$ is the largest index such that $c_\ell > 0$, then 
\begin{equation}
C = J_{n-\ell} (0) \oplus \hat{C} = 
\kbordermatrix{
 & & & n - \ell + 1 & n - \ell + 2 & \cdots & n \\
 & J_{n - \ell}(0) & \vrule &  \\
 \cline{2-7}
n - \ell + 1 & & \vrule & & 1 \\
\vdots & & \vrule & & & \ddots \\
n-1 & & \vrule & & & & 1 \\
n & & \vrule & c_\ell & \ast & \cdots & \ast
 }. \label{compan}
\end{equation} 
The result now follows by applying the Perron-Frobenius theorem for irreducible nonnegative matrices \cite[Theorem 8.4.4]{hj2013} to the matrix $\hat{C}$.

For the converse, we proceed by the contrapositive: to that end, if $d = 0$, then $C = J_n (0)$ and $p$ is clearly not weakly spectrally Perron.
\end{proof}

The following result, attributed to Ostrovsky, provides a sufficient condition that ensures that $p$ is spectrally Perron; we give a proof via nonnegative matrix theory and show that the sufficient condition is also necessary. 

%-----------------------------------------------------------------------------------------------------------------------------------------------------------------------------------------------------------------------------
\begin{thm}[Ostrovosky]
\label{thm:ost}
The polynomial $p$ is spectrally Perron if and only if $d = 1$.
\end{thm}

\begin{proof}
If $d = 1$, then at least one of the elements of $\mathcal{C}$ is positive. We distinguish the following cases:
\begin{enumerate}
[label=\roman*)]
\item $c_n \neq 0$. In this case $C$ is irreducible and since $d = 1$, note that at least two elements of $\mathcal{C}$ are positive (otherwise $d=n$). Moreover, $C$ is primitive since the greatest common divisor of all cycle-lengths of $\dg{C}$ is one. Hence, $p$ is spectrally Perron.  
\item $c_n = 0$. In this case, $C$ is of the form \eqref{compan} and the result follows by applying the exact argument in the previous case to the irreducible matrix $\hat{C}$.  
\end{enumerate}  

To prove necessity, we proceed via the contrapositive; to that end, and without losing generality, assume that $d > 1$ (the case $d=0$ was handled in the proof of \hyperref[thm:cauchy]{Cauchy's Theorem}). We distinguish the following cases:
\begin{enumerate}
[label=\roman*)]
\item $c_n \neq 0$. In this case, $C$ is irreducible and $\rho = \sr{C} \in \sig{C}$, but $C$ has $d$ eigenvalues of modulus $\rho$ (see, e.g., \cite[Corollary 8.4.6(c)]{hj2013}). Thus, $p$ is weakly spectrally Perron but not spectrally Perron. 
\item $c_n = 0$. In this case, $C$ is of the form \eqref{compan} and the result follows by applying the exact argument in the previous case to the irreducible matrix $\hat{C}$.
\end{enumerate}  
From the above, $p$ is spectrally Perron if and only if $d = 1$.
\end{proof}

%-----------------------------------------------------------------------------------------------------------------------------------------------------------------------------------------------------------------------------
\section{Implications for Further Research}
%-----------------------------------------------------------------------------------------------------------------------------------------------------------------------------------------------------------------------------

A real matrix $A$ is called \emph{eventually nonnegative (postive)} if there is a positive integer $k$ such that $A^k$ is nonnegative (respectively, positive). If $p$ is a monic polynomial and $C_p$ is eventually nonnegative, then $p$ is weakly spectrally Perron \cite{f1978}; if $C_p$ is eventually positive, then $p$ is spectrally Perron \cite{h1981, jt2004,n2006}. 

However, as the following example indicates, the companion matrix of a spectrally Perron polynomial need not be eventually nonnegative. 

%-----------------------------------------------------------------------------------------------------------------------------------------------------------------------------------------------------------------------------
\begin{ex}
If $p(t) = t^3 - 2t^2 - t + 2$, then
\[ C = C_p = 
\begin{bmatrix}
0 & 1 & 0 \\
0 & 0 & 1 \\
-2 & 1 & 2
\end{bmatrix}, \]
$\sig{C} = \{ 2,1, -1\}$, but the first column of $C^k$ is negative for every $k \geq 3$. 
\end{ex}

The previous example leads to the following problem.

\begin{prob}
Find necessary and sufficient conditions on the coefficients of a monic polynomial $p$ such that $p$ is (weakly) spectrally Perron.
\end{prob}

%-----------------------------------------------------------------------------------------------------------------------------------------------------------------------------------------------------------------------------
% Bibliography
%-----------------------------------------------------------------------------------------------------------------------------------------------------------------------------------------------------------------------------

\bibliography{master}
\bibliographystyle{abbrv}

\end{document}